\documentclass[11pt]{article}
\usepackage[a4paper, left=1.5cm, right=1.5cm, top=2.5cm, bottom=2.5cm, headsep=1.2cm]{geometry}               
\usepackage[parfill]{parskip}   
\usepackage{graphicx}
\usepackage{amssymb}
\usepackage{color}
\usepackage{amsfonts}
\usepackage{amsmath} 
\usepackage{amsthm}
\usepackage{empheq}
\usepackage{esint}
\usepackage{epstopdf}
\usepackage{verbatim}
\usepackage{ifthen}
\usepackage{comment}
\usepackage{hyperref}
\usepackage{float}
\usepackage{stmaryrd}

\begin{document}

\title{On the infimum convolution inequalities\\ with improved constants}

\author{Marcin Ma\l ogrosz
\footnote{Institute of Mathematics, Polish Academy of Sciences, Warsaw Poland \newline (malogrosz@impan.pl)}}
\date{}

\newtheorem{defi}{Definition}
\newtheorem{theo}{Theorem}
\newtheorem{prop}{Proposition}
\newtheorem{lem}{Lemma}
\newtheorem{rem}{Remark}
\newcommand{\ov}[1]{\overline{#1}}
\newcommand{\un}[1]{\underline{#1}}
\newcommand{\bsym}[1]{\boldsymbol{#1}}
\newcommand{\n}[1]{\lVert#1\rVert}
\newcommand{\eq}[1]{\begin{align}#1\end{align}}

\providecommand{\bs}{\begin{subequations}}
\providecommand{\es}{\end{subequations}}

\maketitle

\begin{abstract}
\noindent The goal of the article is to improve constants in the infimum convolution inequalities (IC for short) which were introduced by R. Lata{\l}a and J.O. Wojtaszczyk. We show that the exponential distribution satisfies IC with constant $2$ but not with constant $1$, which implies that linear functions are not extremal in Maurey's property $(\tau)$. Using transport of measure we use this result to better constants in the IC inequalities for product symmetric log-concave measures as well as in the Talagrand's two level concentration inequality for the exponential distribution.
\end{abstract}


\textbf{Keywords} infimum convolution inequalities, property $(\tau)$, concentration of measure, log-concave measures

\section{Introduction}

In the seminal paper \cite{Mau} B. Maurey introduced the property $(\tau)$ for a probability measure $\mu$ with a cost function $W$ (see Definition \ref{defi:1}) and established its connections with the concentration of measure phenomenon (see Proposition \ref{concentration}). Later in \cite{LW} R. Lata{\l}a and J.O. Wojtaszczyk showed that if a pair $(\mu,W)$ satisfies property $(\tau)$, where $\mu$ is a symmetric probability measure and $W$ is a convex cost function then $W\leq\Lambda^*_{\mu}$, where $\Lambda^*_{\mu}$ is the Cramer transform of $\mu$. This observation led to the definition of the so called infimum convolution inequality, IC for short. Namely a measure $\mu$ satisfies $IC(\beta)$ if the pair $(\mu,\Lambda^*_{\mu}(\cdot/\beta))$ satisfies property $(\tau)$. Lata{\l}a and Wojtaszczyk proved that the symmetric exponential distribution $d\nu=\frac{1}{2}e^{-|x|}dx$ satisfies $IC(9)$ and used that result to prove that any symmetric product log-concave fully supported probability measure satisfies $IC(48)$. Moreover using the connection of IC with the concentration of measure phenomenon the authors proved the two level concentration inequality for product exponential distribution $\nu^n$ with constants $C_1=18,\ C_2=6\sqrt{2}$, obtained previously with rather large constants by Talagrand in \cite{Tal}.\\
The goal of this paper is to improve constants in the inequalities obtained in $\cite{LW}$. We show that any Gaussian measure satisfies $IC(1)$ (Theorem \ref{tw:4}), while one-sided and symmetric exponential distributions satisfy $IC(2)$ (Theorem \ref{tw:5}) but not $IC(1)$ (Theorem \ref{tw:7}). The latter result comes as a surprise as it shows that linear functions are not extremal in the property $(\tau)$ for the exponential measure. Next we prove that any symmetric product log-concave fully supported probability distribution satisfies $IC(9.61929\ldots)$ (Theorem \ref{tw:10}). Finally we obtain Talagrand's two level concentration inequality with constants $C_1=4, C_2=8$ (Theorem \ref{tw:6}).

\subsection{Notation}
In the whole paper $\mu$ denotes a probability measure on the Euclidean space $\mathbb{R}^n$ with scalar product $\langle x,y\rangle=\sum_{i=1}^nx_{i}y_{i}$. We assume that all functions that are considered are Lebesgue measurable. Moreover we use the following notation

\begin{itemize}
\item For $x,y\in\mathbb{R}$ we put	$x\wedge y=\min\{x,y\}, \ x\vee y=\max\{x,y\}$;
\item For a map $T:\mathbb{R}^n\to\mathbb{R}^k$ we denote by $T_{\#}\mu$ the transport of $\mu$ by $T$ defined by $T_{\#}\mu(A)=\mu(T^{-1}(A));$ 
\item For a nonnegative function $W\colon\mathbb{R}^{n}\to[0;+\infty]$ we denote $B_{W}(t)=\{x\in\mathbb{R}^n: W(x)\leq t\}, \ B_{W}=B_{W}(1);$
\item By $|x|_{p}$ we denote $l_p$ norm on $\mathbb{R}^n$ given by $|x|_{p}=\sqrt[p]{\sum_{i=1}^{n}|x_{i}|^{p}}$. Moreover we put $B_{p}(t)=B_{|\cdot|_{p}}(t)$;
\item $g_{\mu}=d\mu/dx$ - density of measure $\mu$ with respect to the Lebesgue measure;
\item $\mu'=(-Id)_{\#}\mu$ - reflection of measure $\mu$ with respect to the origin;
\item $\overline{\mu}=\mu\ast\mu'$ - convolution of $\mu$ and $\mu'$; 
\item $\mu^{\ast n}$ - $n$-th convolution power, $\mu^{n}=\mu^{\otimes n}$ - $n$-th product power;
\item $\gamma$ - standard Gaussian distribution ($g_{\gamma}(x)=\frac{1}{\sqrt{2\pi}}e^{-x^{2}/2}$);  
\item $\nu_{+}$ - exponential distribution ($g_{\nu_{+}}(x)=e^{-x}\mathbb{I}_{[0;\infty)}(x)$);
\item $\nu=\overline{\nu_{+}}$ - symmetric exponential distribution ($g_{\nu}(x)=\frac{1}{2}e^{-|x|}$).
\end{itemize}

\section{Preliminaries}

\subsection{Infimum convolution and property $(\tau)$}

\begin{defi}[Infimum convolution operator $\boxempty$]
\label{defi:1}
For functions $f,g:\mathbb{R}^n\to(-\infty;\infty]$ the infimum convolution of $f$ and $g$ is
\begin{align}
\label{eq:1}
(f\boxempty g)(x)=\inf\{f(x-y)+g(y): \ y\in\mathbb{R}^n\}.
\end{align} 
\end{defi}

In the next Proposition we collect the properties of the infimum convolution operator. \\

\begin{prop}
\label{f:1}
For functions $f,g,h\colon\mathbb{R}^{n}\to(-\infty;\infty]$ one has
\begin{enumerate}
\item $f\boxempty g=g\boxempty f$ (commutativity).
\item $f\boxempty e=f$, where $
e(x)=\left\{\begin{array}{ll} 0 & x=0\\ \infty & x\neq 0\end{array}\right.$ (existence of neutral element).


\item $(f\boxempty g)\boxempty h=f\boxempty (g\boxempty h)$ (associativity).
\item $f\boxempty0=\inf f$.
\item $f\boxempty g + \inf h\leq f\boxempty (g+h)\leq f\boxempty g + \sup h$.
\item If $f_{n}\rightrightarrows f$ then $g\boxempty f_{n}\rightrightarrows g\boxempty f$ where $\rightrightarrows$ denotes uniform convergence.
\item If $f$ is convex then $(f\boxempty f)(x)= 2f(x/2)$.
\item If $g$ is convex and $f(x)=g(2x)/2$ then $f\boxempty f=g$. 
\end{enumerate}
\end{prop}

The next definition was introduced by B. Maurey in \cite{Mau}.\\

\begin{defi}[Property $(\tau)$]
\label{defi:2}
An ordered pair $(\mu,W)$, where $\mu$ is a probabilty measure on $\mathbb{R}^{n}$ and $W\colon\mathbb{R}^{n}\to[0;\infty]$ is a cost function satisfies property $(\tau)$ if for every bounded function $f$
\begin{align}
\label{eq:2}
\int_{\mathbb{R}^n}{e^{W\boxempty f}d\mu}\int_{\mathbb{R}^n}{e^{-f}d\mu}\leq1.
\end{align}
\end{defi}

The motivation for the Definition $\ref{defi:2}$ comes from the following Proposition from \cite{LW} which connects property $(\tau)$ with the concentration of measure phenomenon (see \cite{Led}).\\

\begin{prop}\label{concentration}
Assume that $(\mu, W)$ satisfies property $(\tau)$ then for every Borel set $A$
\begin{enumerate}
\item $\forall_{t>0}\quad\mu(A+B_{W}(t))\geq\frac{e^{t}\mu(A)}{(e^{t}-1)\mu(A)+1}\geq1-\mu(A)^{-1}e^{-t}$.
\item If $\mu(A)=\nu(-\infty;x]$, then $\forall_{t>0}\quad\mu(A+B_{W}(2t))\geq\nu(-\infty;x+t]$.
\end{enumerate}
\end{prop}

First three parts of the next Proposition are from \cite{Mau}, fourth part is a straightforward consequence of the second part, while the fifth part is a generealization of the result from \cite{Goz}.\\

\begin{prop}
\label{f:2}
Let $\mu, \mu_{1}, \mu_{2}$ be measures on $\mathbb{R}^n,\mathbb{R}^{n_{1}},\mathbb{R}^{n_{2}}$ and let $T\colon\mathbb{R}^{n}\to\mathbb{R}^{k}$. Assume that pairs $(\mu,W),(\mu_{1},W_{1}),(\mu_{2},W_{2})$ satisfy property $(\tau)$. Then
\begin{enumerate}

\item Pair $(\mu_{1}\otimes\mu_{2},W)$ satisfies property $(\tau)$, where $W(x_{1},x_{2})=W_{1}(x_{1})+W_{2}(x_{2})$.

\item If $V\colon\mathbb{R}^{k}\to[0;\infty]$ satisfies for all $x,y\in\mathbb{R}^{n}$ the condition 
\begin{align*}
V(T(x)-T(y))\leq W(x-y), 
\end{align*}
then pair $(T_{\#}\mu,V)$ satisfies property $(\tau)$.

\item If $n_{1}=n_{2}$, then pair $(\mu_{1}*\mu_{2},W_{1}\boxempty W_{2})$ satisfies property $(\tau)$. 

\item If $L\colon\mathbb{R}^n\to\mathbb{R}^k$ is an affine map ($L(x)=Ax+b$, where $A\in\mathbb{R}^{n\times k}$ and $b\in\mathbb{R}^k$) such that
\begin{align*}
V(Ax)\leq W(x), 
\end{align*}
then pair $(L_{\#}\mu,V)$ satisfies property $(\tau)$.

\item If $W(x)=\hat{W}(|x|)$ where $\hat{W}\colon[0;\infty)\to[0;\infty]$ is nondecreasing, then pair\\ 
$(T_{\#}\mu,(\hat{W}\circ\omega_{T})(|\cdot|))$ satisfies property $(\tau)$, where for $h\geq 0$  
\begin{align*}
\omega_{T}(h)=\inf\{|x-y|:|T(x)-T(y)|\geq h\},\quad\inf{\emptyset}=\infty.
\end{align*}

\end{enumerate}
\end{prop}

\subsection{Transforms}

In this section we recall definitions and basic properties of Laplace, Legendre and Cramer transforms. These operators are used in convex analysis (see \cite{MagTik}) and in the theory of large deviations (see \cite{DeuStr}).\\ 

\begin{defi}[Laplace transform]
Laplace transform of a measure $\mu$ on $\mathbb{R}^n$ is
\begin{align*}
M_{\mu}(x)=\int_{\mathbb{R}^n}{e^{<x,y>}}d\mu(y).
\end{align*}
\end{defi}

\begin{prop}
For probability distributions $\mu,\mu_{1},\mu_{2}$ on $\mathbb{R}^{n},\mathbb{R}^{n_{1}},\mathbb{R}^{n_{2}}$ and $x\in\mathbb{R}^{n},x_{1}\in\mathbb{R}^{n_{1}},x_{2}\in\mathbb{R}^{n_{2}}$ one has

\begin{enumerate}
\item $M_{\mu_{1}\otimes\mu_{2}}(x_{1},x_{2})=M_{\mu_{1}}(x_{1})M_{\mu_{2}}(x_{2})$.
\item If $T(x)=Ax+b$, where $A\in\mathbb{R}^{m\times n}$, $b\in\mathbb{R}^m$ then for $y\in\mathbb{R}^m$ there is $M_{T_{\#}\mu}(y)=e^{\langle y,b\rangle}M_{\mu}(A^{t}y)$. 
\item If $n_{1}=n_{2}=n$, then $M_{\mu_{1}*\mu_{2}}(x)=M_{\mu_{1}}(x)M_{\mu_{2}}(x)$.
\end{enumerate}
\end{prop}
\ \\
\begin{defi}[Legendre transform]
Legendre transform of a function $f\colon\mathbb{R}^n\to(-\infty;\infty]$ is
\begin{align*}
f^{\ast}(x)=\sup_{y}\{\langle x,y\rangle-f(y)\}.
\end{align*}
\end{defi}

\begin{prop}
For arbitrary $f,g\colon\mathbb{R}^{n}\to(-\infty;\infty]$

\begin{enumerate}
\item $f^{*}$ is convex.
\item $f^{\ast\ast}\leq f$.
\item If $f$ is convex and lower semicontinuous then $f^{\ast\ast}=f$.
\item If $f\leq g$, then $f^{\ast}\geq g^{\ast}$.
\item If $C$ is a real number then $(Cf)^{\ast}(x)=Cf^{\ast}(x/C)$ and $(f(\cdot/C))^{\ast}(x)=f^{\ast}(Cx)$.
\item If $f,g$ are convex then $(f\boxempty g)^{\ast}=f^{\ast}+g^{\ast}$.
\end{enumerate}
\end{prop}
\ \\
\begin{defi}[Cramer transform]
Cramer transform of a measure $\mu$ is $\Lambda^{\ast}_{\mu}$, where $\Lambda_{\mu}=\ln M_{\mu}$.
\end{defi}
\ \\
\begin{prop}
For probability distributions $\mu,\mu_{1},\mu_{2}$ on $\mathbb{R}^{n},\mathbb{R}^{n_{1}},\mathbb{R}^{n_{2}}$ and $x\in\mathbb{R}^{n},x_{1}\in\mathbb{R}^{n_{1}},x_{2}\in\mathbb{R}^{n_{2}}$ one has

\begin{enumerate}
\item $\Lambda_{\mu_{1}\otimes\mu_{2}}(x_{1},x_{2})=\Lambda_{\mu_{1}}(x_{1})+\Lambda_{\mu_{2}}(x_{2})$.
\item $\Lambda_{\mu_{1}\otimes\mu_{2}}^{*}(x_{1},x_{2})=\Lambda_{\mu_{1}}^{*}(x_{1})+\Lambda_{\mu_{2}}^{*}(x_{2})$
\item If $T(x)=Ax+b$, where $A\in\mathbb{R}^{m\times n}$, $b\in\mathbb{R}^m$, then $\Lambda_{T_{\#}\mu}(y)=\langle y,b\rangle+\Lambda_{\mu}(A^{t}y)$ for $y\in\mathbb{R}^m$.
\item If $n_{1}=n_{2}=n$, then $\Lambda_{\mu_{1}*\mu_{2}}(x)=\Lambda_{\mu_{1}}(x)+\Lambda_{\mu_{2}}(x)$.
\item $\Lambda_{\mu}$ is convex.
\item $\Lambda^{\ast}_{\mu}$ is convex and nonnegative.
\item $\Lambda^{\ast}_{\mu}(0)=0$.
\item If $\mu$ is a symmetric probability measure, then $\Lambda^{*}_{\mu}$ is even and $\Lambda^{*}_{\overline{\mu}}(x)=2\Lambda^{*}_{\mu}(x/2)$.
\end{enumerate}
\end{prop}

\subsection{Infimum convolution inequality - IC}

The next Proposition which was proved in \cite{LW} gives an upper bound for any convex cost function from the Definition \ref{defi:2}.\\

\begin{prop}
If a pair $(\mu,W)$ satisfies property $(\tau)$ and $W$ is convex, then $W\leq\Lambda^{\ast}_{\overline{\mu}}$.
\end{prop}

It motivates the following definition\\

\begin{defi}[Infimum convolution inequality - IC]
A probability measure $\mu$ on $\mathbb{R}^{n}$ satisfies the infimum convolution inequality with constant $\beta>0$ if the pair $(\mu,\Lambda^{\ast}_{\overline{\mu}}(\cdot/\beta))$ satisfies property $(\tau)$.
\end{defi}

In the next Proposition we collect properties of the infimum convolution inequalities\\  

\begin{prop}
\
\label{fakt:6}
For any probability measures $\mu,\mu_{1},\mu_{2}$ on $\mathbb{R}^{n},\mathbb{R}^{n_{1}},\mathbb{R}^{n_{2}}$ satisfying $\rm{IC}(\beta), \rm{IC}(\beta_{1}), \rm{IC}(\beta_{2})$ there holds
\begin{enumerate}
\item If $L$ is an affine map, then $L_{\#}\mu$ satisfies $\rm{IC}(\beta)$.
\item The product $\mu_{1}\otimes\mu_{2}$ satisfies $\rm{IC}(\beta_{1}\vee\beta_{2})$. 
\item
If $n_{1}=n_{2}$, then convolution $\mu_{1}*\mu_{2}$ satisfies $\rm{IC}(\beta_{1}\vee\beta_{2})$.
\item Symmetrization $\overline{\mu}$ satisfies $\rm{IC}(\beta)$.
\end{enumerate}
\end{prop}

\begin{proof}
\
\begin{enumerate}
\item Denote $L(x)=Ax+b$. Thanks to part 4 of Proposition \ref{f:2} it is enough to check that
\begin{align*}
\Lambda^{*}_{\overline{L_{\#}\mu}}\Big(\frac{Ax}{\beta}\Big)\leq\Lambda^{*}_{\overline{\mu}}\Big(\frac{x}{\beta}\Big).
\end{align*}
We will show that
\begin{align*}
\overline{L_{\#}\mu}=A_{\#}\overline{\mu}.
\end{align*}
Indeed if $X_{1},X_{2}$ are independent random variables with distribution $\mu$ then $A(X_{1}-X_{2})$ has distribution $A_{\#}\overline{\mu}$ while $L(X_{1})-L(X_{2})$ has distribution $\overline{L_{\#}\mu}$ and the desired equality follows from $A(X_{1}-X_{2})=L(X_{1})-L(X_{2})$. Thus
\begin{align*}
\Lambda^{*}_{\overline{L_{\#}\mu}}(Ax)&=
\Lambda^{*}_{A_{\#}\overline{\mu}}(Ax)
=\sup_{y}\Big\{\langle Ax,y\rangle-\ln\int{e^{\langle y,Az
\rangle}d\overline{\mu}(z)}\Big\}\\
&=\sup_{y}\Big\{\langle x,A^{t}y\rangle-\ln\int{e^{\langle A^{t}y,z \rangle}d\overline{\mu}(z)}\Big\}\leq\sup_{y}\Big\{\langle x,y\rangle-\ln\int{e^{\langle y,z \rangle}d\overline{\mu}(z)}\Big\}
\\
&=\Lambda^{*}_{\overline{\mu}}(x),
\end{align*}
hence
\begin{align*}
\Lambda^{*}_{\overline{L_{\#}\mu}}\Big(\frac{Ax}{\beta}\Big)=\Lambda^{*}_{\overline{L_{\#}\mu}}\Big(A\Big(\frac{x}{\beta}\Big)\Big)\leq\Lambda^{*}_{\overline{\mu}}\Big(\frac{x}{\beta}\Big).
\end{align*}

\item Since for $i=1,2$ pair $(\mu_{i},\Lambda^{*}_{\overline{\mu_{i}}}(\cdot/\beta_{i}))$ has $(\tau)$ property, so using Proposition \ref{f:2} the pair $(\mu_{1}\otimes\mu_{2},W)$ has property $(\tau)$, where 
\begin{align*}
W(x_{1},x_{2})=\Lambda^{*}_{\overline{\mu_{1}}}\Big(\frac{x_{1}}{\beta_{1}}\Big)+\Lambda^{*}_{\overline{\mu_{2}}}\Big(\frac{x_{2}}{\beta_{2}}\Big).
\end{align*}
The claim follows from
\begin{align*}
\Lambda^{*}_{\overline{\mu_{1}\otimes\mu_{2}}}\Bigg(\frac{(x_{1},x_{2})}{\beta_{1}\vee\beta_{2}}\Bigg)=\Lambda^{*}_{\overline{\mu_{1}}\otimes\overline{\mu_{2}}}\Bigg(\frac{(x_{1},x_{2})}{\beta_{1}\vee\beta_{2}}\Bigg)\leq\Lambda^{*}_{\overline{\mu_{1}}}\Big(\frac{x_{1}}{\beta_{1}}\Big)+\Lambda^{*}_{\overline{\mu_{2}}}\Big(\frac{x_{2}}{\beta_{2}}\Big)=W(x_{1},x_{2}).
\end{align*}
\item $\mu_{1}*\mu_{2}=T_{\#}(\mu_{1}\otimes\mu_{2})$ for $T(x_{1},x_{2})=x_{1}+x_{2}$ so it is enough to use Part 1 and Part 2.
\item $\overline{\mu}=S_{\#}(\mu\otimes\mu)$, for $S(x_{1},x_{2})=x_{1}-x_{2}$.
\end{enumerate}
\end{proof}

\section{Results}

\subsection{IC for Gaussian and exponential distributions}

We start by analizing the Gaussian distributions.\\ 

\begin{theo}
\label{tw:4}
Every Gaussian distribution on $\mathbb{R}^n$ satisfies $\rm{IC}(1)$. 
\end{theo}

\begin{proof}
Since any Gaussian distribution is an affine transport of $\gamma^{n}$, thus using  Proposition \ref{fakt:6} it suffices to prove that $\gamma$ satisfies $\rm{IC}(1)$.
Standard calculations show that $\Lambda_{\gamma}^{\ast}(x)=x^{2}/2$.
Since $\gamma$ is a symmetric distribution thus $\Lambda_{\overline{\gamma}}^{\ast}(x)=2\Lambda_{\gamma}^{\ast}(x/2)=x^{2}/4$. The claim follows from the fact that for $G(x)=x^{2}/4$ the pair $(\gamma,G)$ satisfies property $(\tau)$ which was shown in \cite{Mau}.
\end{proof}

Next we turn our attention to exponential distributions.\\

\begin{theo}
\label{tw:5}
\
\begin{enumerate}
\item One sided exponential distribution $\nu_{+}$ satisfies $\rm{IC}(2)$.
\item Symmetric exponential distribution $\nu$ satisfies $\rm{IC}(2)$.
\end{enumerate}
\end{theo}

In the proof of Theorem \ref{tw:5} we will use the following
\ \\
\begin{lem}
\label{lem:5.1}
If a function $W\geq0$ satisfies the following two conditions
\begin{enumerate}
\item $2|W'|\leq 1$,
\item $e^{W}(1-4(W')^2)\geq1$,
\end{enumerate}
then the pair $(\nu_{+},W)$ satisfies property $(\tau)$.
\end{lem}

\begin{proof}
In \cite{Mau} it was shown that the pair $(\nu,U)$ satisfies property $(\tau)$, where
\begin{align*}
U(x)=\left\{
\begin{array}{ll}
\frac{1}{36}x^2 &, |x|\leq4
\\
\frac{2}{9}(|x|-2) &, |x|>4
\end{array}
\right..
\end{align*}
From that proof it follows that conditions given in the Lemma are sufficient for the pair $(\nu_+,W)$ to have the property $(\tau)$.
\end{proof}

\begin{proof}[Proof of Theorem \ref{tw:5}]
To prove the first part we need to show that
\begin{align*}
W(x)=\Lambda^{\ast}_{\overline{\nu_{+}}}(x/2)=\Lambda^{\ast}_{\nu}(x/2)
\end{align*}
satisfies conditions given in the Lemma \ref{lem:5.1}. 
Denote 
\begin{align*}
\phi(x)=W(2x)=\Lambda^{\ast}_{\nu}(x). 
\end{align*}
We calculate
\begin{align*}
\phi(x)=\sqrt{x^2+1}-1-\ln{\Big(\frac{\sqrt{x^{2}+1}+1}{2}\Big)},\quad \phi'(x)=\frac{x}{\sqrt{x^{2}+1}+1}, 
\end{align*}
from which $2|W'(x)|=|\phi'(x/2)|\leq1$. \\
The second condition of Lemma \ref{lem:5.1} follows from the following estimation:
\begin{align*} e^{W(2x)}(1-4(W'(2x))^2)&=e^{\phi(x)}(1-(\phi'(x))^2)=e^{\sqrt{x^{2}+1}-1}\frac{2}{\sqrt{x^{2}+1}+1}\Big(1-\Big(\frac{x}{\sqrt{x^{2}+1}+1}\Big)^2\Big)
\\
&=e^{\sqrt{x^{2}+1}-1}\frac{4}{(\sqrt{x^{2}+1}+1)^2}=\Big(\frac{e^y}{y+1}\Big)^2\geq1, 
\end{align*}
where $y=\frac{\sqrt{x^{2}+1}-1}{2}$.   
\\
The second part of the Theorem \ref{tw:5} is a consequence of $\nu=\overline{\nu_{+}}$ and the fourth part of Proposition \ref{fakt:6}.
\end{proof}

\begin{rem}
\noindent It can be shown that $\Lambda^{*}_{\overline{\nu}}(x/2)>U(x)$ and thus Theorem \ref{tw:5} improves the result from \cite{Mau}.
\end{rem}

The next Theorem gives a negative answer to the hypothesis that for the exponential distribution linear functions are extremal in the property $(\tau)$.\\

\begin{theo}
\
\label{tw:7}
\begin{enumerate}
\item  $\nu_{+}$ does not satisfy $\rm{IC}(1)$.
\item $\nu$ does not satisfy $\rm{IC}(1)$.
\end{enumerate}
\end{theo}
\ \\
To prove Theorem \ref{tw:7} we will use the following\\
\begin{lem}
\label{lem:5.2}
If $\mu$ satisfies $\rm{IC}(1)$, then 
\begin{align*}
\int{e^{2\Lambda^{*}_{\overline{\mu}}(x/2)}d\mu(x)}\int{e^{-\Lambda^{*}_{\overline{\mu}}(x)}d\mu(x)}\leq 1.
\end{align*}
\end{lem}

\begin{proof}
We substitute $f=W=\Lambda^{*}_{\overline{\mu}}$ in the Definition \ref{defi:2} and use the identity $f\boxempty W=W\boxempty W=2W(\cdot/2)$, which is a consequence of convexity of $W$ and Part 7 of Proposition \ref{f:1}). 
\end{proof}

\begin{proof}[Proof of Theorem \ref{tw:7}]
Using Lemma \ref{lem:5.2} it is enough to show that
\begin{align*}
\int{e^{2\Lambda^{*}_{\overline{\nu_{+}}}(x/2)}d\nu_{+}(x)}\int{e^{-\Lambda^{*}_{\overline{\nu_{+}}}(x)}d\nu_{+}(x)}> 1,
\end{align*}
and
\begin{align*}
\int{e^{2\Lambda^{*}_{\overline{\nu}}(x/2)}d\nu(x)}\int{e^{-\Lambda^{*}_{\overline{\nu}}(x)}d\nu(x)}>1.
\end{align*}
Denote
\begin{align*}
f(x)=\Lambda^{*}_{\nu}(x)=\sqrt{x^2+1}-1-\ln{\Big(\frac{\sqrt{x^{2}+1}+1}{2}\Big)}.
\end{align*}
Then
\begin{align*}
\Lambda^{*}_{\overline{\nu_{+}}}(x)=\Lambda^{*}_{\nu}(x)=f(x),\quad
\Lambda^{*}_{\overline{\nu}}(x)=2\Lambda^{*}_{\nu}(x/2)=2f(x/2),
\end{align*}
which after change of variables is equivalent to 
\begin{align*}
2\int_{0}^{\infty}{e^{2(f(y)-y)}dy}\int_{0}^{\infty}{e^{-(f(y)+y)}dy}>1,
\\
8\int_{0}^{\infty}{e^{4(f(y)-y)}dy}\int_{0}^{\infty}{e^{-2(f(y)+y)}dy}>1.
\end{align*}

The above inequalities where verified using numerical integration in Mathematica software. We include the computations
\\
\\
\noindent\(\pmb{f[\text{y$\_$}]\text{:=}\text{Sqrt}[1+y{}^{\wedge}2]-1-\text{Log}[(\text{Sqrt}[1+y{}^{\wedge}2]+1)/2]}\)

\noindent\(\pmb{\text{I1}=\text{NIntegrate}[\text{Exp}[2(f[y]-y)],\{y,0,\text{Infinity}\}]}\)

\noindent\(0.822119\)

\noindent\(\pmb{\text{I2}=\text{NIntegrate}[\text{Exp}[-(f[y]+y)], \{y,0,\text{Infinity}\}]}\)

\noindent\(0.787272\)

\noindent\(\pmb{2*\text{I1}*\text{I2}}\)

\noindent\(1.29446\)

\noindent\(\pmb{\text{I3}=\text{NIntegrate}[\text{Exp}[4(f[y]-y)],\{y,0,\text{Infinity}\}]}\)

\noindent\(0.29795\)

\noindent\(\pmb{\text{I4}=\text{NIntegrate}[\text{Exp}[-2(f[y]+y)],\{y,0,\text{Infinity}\}]}\)

\noindent\(0.426799\)

\noindent\(\pmb{8*\text{I3}*\text{I4}}\)

\noindent\(1.01732\)

\end{proof}

\subsection{IC for log-concave distributions}

The next Theorem deals with the behaviour of IC under transport of measure by certain special class of maps

\begin{prop}
\label{tw:8}
Assume that $T\colon\mathbb{R}\to\mathbb{R}$ satisfies the following conditions 
\begin{enumerate}
\item is odd and nondecreasing, 
\item is concave on $[0;\infty)$, 
\item there exists finite $T'(0)$, 
\item $\int{x^{2}dT_{\#}\nu}=1$. 
\end{enumerate}
If $c\geq T'(0)$ then measure $T_{\#}\nu$ satisfies $\rm{IC}(2c\delta)$, where $\delta>0$ and $\Lambda^{*}_{\nu}(\delta)=\ln2+1/c$. 
\end{prop}

To prove Proposition \ref{tw:8} we will use four Lemmas. Lemmas \ref{lem:8.2} and \ref{lem:8.3} were proved in \cite{LW}.
\ \\
\begin{lem}
\label{lem:8.1}
Under the assumptions of Proposition \ref{tw:8} one has $\omega_{T}(2T(x))=2x$ for $x\geq0$.
\end{lem}

\begin{proof}
We will show first that
\begin{align*}
|T(x)-T(y)|\leq2T\Big(\frac{|x-y|}{2}\Big).
\end{align*}
Without loss of generality we may assume that $x\geq y$. \\
If $x\geq y\geq0$ then
\begin{align*}
2T\Big(\frac{|x-y|}{2}\Big)=2T\Big(\frac{x-y}{2}\Big)\geq2\Big(\frac{T(x)}{x}\frac{x-y}{2}\Big)=\frac{T(x)}{x}(x-y)\geq T(x)-T(y)=|T(x)-T(y)|.
\end{align*} 
The first inequality follows from the fact that the graph of the conc$T$ lies above the line passing through points $(0,T(0))$, $(x,T(x))$. The second inequality is a consequence of the that the gradient of the line passing through points $(0,T(0)),(x,T(x))$ is larger than the gradient of the line passing through points $(y,T(y)),(x,T(x))$.\\
If $0\geq x\geq y$, then $-y\geq-x\geq0$, so using the previous case one gets
\begin{align*}
|T(x)-T(y)|=|T(-y)-T(-x)|\leq2T\Big(\frac{|-y+x|}{2}\Big)=2T\Big(\frac{|x-y|}{2}\Big).
\end{align*}
If $x\geq 0\geq y$, then
\begin{align*}
|T(x)-T(y)|=2\Big(\frac{1}{2}T(x)+\frac{1}{2}T(-y)\Big)\leq2T\Big(\frac{x-y}{2}\Big)=2T\Big(\frac{|x-y|}{2}\Big).
\end{align*}
To finish the proof let us observe that
\begin{align*}
\omega_{T}(2T(x))&=\inf\{|x'-y'|: |T(x')-T(y')|\geq2T(x)\}\geq\inf\{|x'-y'|: \  2T\Big(\frac{|x'-y'|}{2}\Big)\geq2T(x)\}
\\
&=\inf\{|x'-y'|:|x'-y'|\geq2x\}\geq2x,
\end{align*}
and the equality holds for $x'=-y'=x$.
\end{proof}


\begin{lem}
\label{lem:8.2}
If $\mu$ is a symmetric, probability measure on $\mathbb{R}$ such that $\int{x^{2}d\mu(x)}=1$ then for $0\leq x\leq 1$ the following holds
\begin{align*}
\Lambda^{*}_{\mu}(x)\leq(\ln(\cosh))^{*}(x)=\frac{1}{2}[(1+x)\ln(1+x))+(1-x)\ln(1-x)]. 
\end{align*}
\end{lem}
\ \\

\begin{lem}
\label{lem:8.3}
If $\mu$ is a symmetric, probability measure on $\mathbb{R}$ then
$\Lambda^{*}_{\mu}(x)\leq-\ln(\mu[x;\infty))$. 
\end{lem}
\ \\

\begin{lem}
\label{lem:8.4}
Cramer transform of the symmetric exponential distribution $\nu$ satisfies for $0\leq x\leq 1$ 
\begin{align*}
\Lambda^{*}_{\nu}(\theta x)\geq(\ln(\cosh))^{*}(x), 
\end{align*}
where $\theta>0$ is such that
$\Lambda^{*}_{\nu}(\theta)=(\ln(\cosh))^{*}(1)=\ln2$.
\end{lem}

\begin{proof}
For $0\leq x\leq1$ define 
\begin{align*}
H(x)&=\Lambda_{\nu}^{*}(\theta x)-(\ln(\cosh))^{*}(x)
\\
&=\sqrt{1+\theta^{2}x^{2}}-1-\ln\Big(\frac{\sqrt{1+\theta^{2}x^2}+1}{2}\Big)-\frac{1}{2}[(1+x)\ln(1+x)+(1-x)\ln(1-x)].
\end{align*}
From standard calculations we get
\begin{align*}
H'(x)=\frac{\theta^{2}x}{\sqrt{1+\theta^{2}x^2}+1}-\frac{1}{2}\ln\Big(\frac{1+x}{1-x}\Big),\quad 
H''(x)=\frac{\theta^{2}}{\sqrt{1+\theta^{2}x^2}+1+\theta^{2}x^2}-\frac{1}{1-x^2}.
\end{align*}
Since $\theta>\sqrt{2}$ (because $\ln2=\Lambda^{*}_{\nu}(\theta)>\Lambda^{*}_{\nu}(\sqrt{2})$) we check that
$H''(x)>0$ for $x\in[0;x_{0})$ and $H''(x)<0$ for $x\in(x_{0},1)$, where $x_{0}=\sqrt{\frac{4\theta^{2}-3-\sqrt{8\theta^{2}+9}}{8\theta^{2}}}$. Since $H(0)=H'(0)=0$ we conlcude that $H$ is increasing on $[0;x_{0}]$, in particular $H\geq0$ on $[0;x_{0})$. Inequality $H\geq 0$ on $[x_{0};1]$ follows from $H(x_{0})\geq0$, $H(1)=0$ and concavity of $H$ on $[x_{0};1]$.  
\end{proof}

\begin{proof}[Proof of Proposition \ref{tw:8}]
Using Part 2 of Theorem \ref{tw:5} the pair $(\nu,\Lambda^{*}_{\overline{\nu}}(\cdot/2))$ has property $(\tau)$. Hence using Part 5 of Proposition \ref{f:2} the pair $(T_{\#}\nu,\Lambda^{*}_{\overline{\nu}}(\omega_{T}(|\cdot|)/2))$ has property $(\tau)$. To finish the proof it is enough to show that for $\beta=2c\delta$ there is
\begin{displaymath}
\Lambda^{*}_{\overline{T_{\#}\nu}}\Big(\frac{y}{\beta}\Big)\leq\Lambda^{*}_{\overline{\nu}}\Big(\frac{\omega_{T}(|y|)}{2}\Big).
\end{displaymath}
Since functions which are present in the above inequality are even we can assume without loss of generality that $y\geq0$. Due to the symmetry of measures
 $\nu$ and $T_{\#}\nu$ the inequality is equivalent to
\begin{displaymath}
\Lambda^{*}_{T_{\#}\nu}\Big(\frac{y}{2\beta}\Big)\leq\Lambda^{*}_{\nu}\Big(\frac{\omega_{T}(y)}{4}\Big).
\end{displaymath}
Let us observe that if $y\notin 2T(\mathbb{R})=\{2T(x):x\in\mathbb{R}\}$, then 
$\{(x',y'):|T(x')-T(y')|\geq y\}=\emptyset$ hence $\omega_{T}(y)=\infty$ and the inequality is true. If $y\in 2T(\mathbb{R})$ then $y=2T(x)$, so due to Lemma \ref{lem:8.1} it is enough to show that for $x\geq0$ the following inequality holds
\begin{displaymath}
\Lambda^{*}_{T_{\#}\nu}\Big(\frac{T(x)}{\beta}\Big)\leq\Lambda^{*}_{\nu}\Big(\frac{\omega_{T}(2T(x))}{4}\Big)=\Lambda^{*}_{\nu}\Big(\frac{x}{2}\Big).
\end{displaymath}
We consider two cases\\
\ \\
\textbf{Case 1.} ($0\leq\frac{cx}{\beta}\leq1$)\\
Using concavity of $T$ on $[0;\infty)$ and $T(0)=0$ we have
\begin{align*}
\frac{T(x)}{\beta}=\frac{1}{\beta}T(x)+\Big(1-\frac{1}{\beta}\Big)T(0)\leq T\Big(\frac{x}{\beta}\Big)\leq c\frac{x}{\beta},
\end{align*}
thus using Lemma \ref{lem:8.2} and Lemma \ref{lem:8.4} we get
\begin{align*}
\Lambda^{*}_{T_{\#}\nu}\Big(\frac{T(x)}{\beta}\Big)\leq\Lambda^{*}_{T_{\#}\nu}\Big(\frac{cx}{\beta}\Big)\leq(\ln\cosh)^{*}\Big(\frac{cx}{\beta}\Big)\leq \Lambda^{*}_{\nu}\Big(\theta\frac{cx}{\beta}\Big)\leq \Lambda^{*}_{\nu}\Big(\frac{x}{2}\Big),
\end{align*}
since 
\begin{align*}
\theta\frac{c}{\beta}=\frac{\theta}{2\delta}\leq\frac{1}{2}.
\end{align*}
\textbf{Case 2.} ($\frac{cx}{\beta}\geq1$)\\
From the fact that $\Lambda^{*}_{T_{\#}\nu}$ is nondecreasing on $[0;\infty)$ (because it is even and convex) and Lemma \ref{lem:8.3} we have
\begin{align*}
\Lambda^{*}_{T_{\#}\nu}\Big(\frac{T(x)}{\beta}\Big)&\leq\Lambda^{*}_{T_{\#}\nu}\Big(T\Big(\frac{x}{\beta}\Big)\Big)\leq h_{T_{\#}\nu}\Big(T\Big(\frac{x}{\beta}\Big)\Big)=-\ln\Big(T_{\#}\nu\Big[T\Big(\frac{x}{\beta}\Big);\infty\Big)\Big)\\
&=-\ln\Big(\nu\Big[\frac{x}{\beta};\infty\Big)\Big)
=\frac{x}{\beta}+\ln2.
\end{align*}
To finish the proof it suffices to show that for $x\geq\frac{\beta}{c}$
\begin{align}
\frac{x}{\beta}+\ln2\leq \Lambda^{*}_{\nu}\Big(\frac{x}{2}\Big).\label{ineq}
\end{align}
Denote $a(x)=\frac{x}{\beta}+\ln2$ and $b(x)=\Lambda^{*}_{\nu}\Big(\frac{x}{2}\Big)$. Then \eqref{ineq} is a consequence of the fact that $a$ is an affine function, $b$ is convex and increasing and 
\begin{align*}
a(0)=\ln2>0=b(0),\quad a(\beta/c)=1/c+\ln2=\Lambda_{\nu}^{*}(\delta)=b(\beta/c).
\end{align*}
\end{proof}


\begin{defi}[Logarithmically concave measure]
\label{def:7}
We call a measure $\mu$ on $\mathbb{R}^{n}$ \em{logarithmically concave} (\em{log-concave}) if for any nonempty compact sets $A,B$ and $t\in[0;1]$,
\begin{align*}
\mu(tA+(1-t)B)\geq\mu(A)^{t}\mu(B)^{1-t}.
\end{align*}
\end{defi}

The following Proposition (see \cite{Bor}) gives a full characterisation of log-concave measures with a fully dimensional support.\\

\begin{prop}
\label{tw:9}
A measure $\mu$ on $\mathbb{R}^{n}$ with fully dimenstional support (i.e. there does not exist a proper affine subspace cotaining the support of the measure) is log-concave if and only if it is absolutely continuous with respect to the Lebesgue measure and has a log-concave density ($g_{\mu}(x)=e^{-W(x)}$, where $W\colon\mathbb{R}^n\to(-\infty;\infty]$ is convex). 
\end{prop}

The next Theorem was proved in \cite{LW}. Our proof improves the constant significantly.\\

\begin{theo}
\label{tw:10}
Every symmetric, product, log-concave probability measure on $\mathbb{R}^n$ with fully dimensional support satisfies $IC(C)$ with a universal constant $C=2\sqrt{3}\delta\approx9.61929\ldots$, where $\delta>0$ is such that $\Lambda^{*}_{\nu}(\delta)=\ln2+1/\sqrt{3}$.
\end{theo}

To prove Theorem \ref{tw:10} we will use the following Proposition which is a modification of the result obtained by Hensley (see \cite{Hen}).\\

\begin{prop}
\label{l:10.1}
If $g\colon\mathbb{R}\to[0;\infty)$ is even, nonincreasing on $[0;\infty)$ and satisfies 
\begin{enumerate}
\item $\int{g(x)dx}=1$,
\item $\int{x^{2}g(x)dx}=1,$ 
\end{enumerate}
then $g(0)\geq\frac{1}{2\sqrt{3}}$.
\end{prop}

\begin{proof}
For $c>0$ we denote by $A(c)$ the set of functions $g\colon\mathbb{R}\to[0;\infty)$ such that
\begin{enumerate}
\item $g$ is even, nonincreasing on $[0;\infty)$;
\item $g(0)=c$;
\item $\int{g(x)dx}=1$.
\end{enumerate}
\noindent We will find $m(c)=\inf\{\int{x^{2}g(x)dx}: g\in A(c)\}$. Denote $u(x)=c\mathbb{I}_{[-1/(2c);1/(2c)]}(x)$. Observe that $u\in A(c)$. Moreover for any function $g\in A(c)$, using integration by parts we obtain
\begin{align*}
\int{x^{2}g(x)dx}&=2\int_{0}^{\infty}{x^{2}g(x)dx}=
2\int_{0}^{\infty}{x^{2}\Big(\int_{x}^{\infty}{-g(s)ds}\Big)'dx}=
2\int_{0}^{\infty}{(x^{2})'\Big(\int_{x}^{\infty}{g(s)ds}\Big)dx}
\\
&=4\int_{0}^{\infty}{x\Big(\frac{1}{2}-\int_{0}^{x}{g(s)ds}\Big)dx}
\geq4\int_{0}^{\infty}{x\Big(\frac{1}{2}-\int_{0}^{x}{u(s)ds}\Big)dx}=\int{x^{2}u(x)dx}=\frac{1}{12c^2},
\end{align*} 
hence $m(c)=\frac{1}{12c^2}$. Assume now that $g$ satisfies the assumptions of the Proposition. Then $g\in A(g(0))$ and
\begin{align*}
1=\int{x^{2}g(x)dx}\geq m(g(0))=\frac{1}{12(g(0))^{2}},
\end{align*}
which finishes the proof.
\end{proof}

\begin{proof}[Proof of Theorem \ref{tw:10}]
Using Proposition \ref{fakt:6} one can assume that $\mu$ is one dimensional and isotropic (i.e. $\int x^2d\mu=1$). Using Proposition \ref{tw:9} the density of $\mu$ is $g_{\mu}(x)=e^{-W(x)}$, for certain even, convex function $W$. Let $T\colon\mathbb{R}\to\mathbb{R}$ be the increasing reaarangement transporting $\nu$ to $\mu$ i.e.
$T=F^{-1}_{\mu}\circ F_{\nu}$, where $F_{\nu}$ and $F_{\mu}$ are cummulative distribution functions. Then $T$ is nondecreasing, odd and concave on $[0;\infty)$ and $T'(0)=1/(2g_{\mu}(0))\leq\sqrt{3}$, where the last inequality follows from Proposition \ref{l:10.1}. Thus $T$ fulfills assumptions of Theorem \ref{tw:8} with constant $c=\sqrt{3}$ which finishes the proof.


\end{proof}

\subsection{Talagrand's two level concentration inequality for exponential distribution}

The next theorem with rather large constants goes back to Talagrand (see \cite{Tal}). The same result with better constants ($C_{1}=18, C_{2}=6\sqrt{2}$) was obtained in \cite{LW}. The proof that we present improves them even further.\\

\begin{theo}
\label{tw:6}
There exist constants $C_{1},C_{2}$ such that for every $n\geq1$ and Borel set $A\subset\mathbb{R}^n$,
\begin{displaymath}
\nu^{n}(A)=\nu(-\infty;x]\Longrightarrow\forall_{t\geq0}\quad\nu^{n}(A+C_{1}tB_{1}^{n}+C_{2}\sqrt{t}B_{2}^{n})\geq\nu(-\infty;x+t],
\end{displaymath}
moreover one can put $C_{1}=4, \ C_{2}=8.$
\end{theo}


In the proof of Theorem \ref{tw:6} we will use two lemmas\\
\begin{lem}
\label{lem:6.1}
Assume that $W\colon\mathbb{R}\to[0;\infty]$ satisfies for certain constants $a,C_{1},C_{2}>0$
\begin{align*}
\forall_{t>0}\quad B_{W}(at)\subset C_{1}tB_{1}^{1}+C_{2}\sqrt{t}B_{2}^{1},
\end{align*}
then for every $n\geq 1$ one has
\begin{align*}
\forall_{t>0}\quad B_{W_{n}}(at)\subset C_{1}tB_{1}^{n}+C_{2}\sqrt{t}B_{2}^{n},
\end{align*}
where $W_{n}(x)=\sum_{i=1}^{n}W(x_{i})$.
\end{lem}

\begin{proof}
Let $n\geq1, t>0$ and $x\in B_{W_{n}}(at)$. Observe that
\begin{align*}
x_{i}\in B_{W}\Big(a\frac{W(x_{i})}{a}\Big)\subset C_{1}\frac{W(x_{i})}{a}B_{1}^{1}+C_{2}\sqrt{\frac{W(x_{i})}{a}}B_{2}^{1}.
\end{align*} 
Thus $x_{i}=y_{i}+z_{i}$, where $|y_{i}|\leq C_{1}\frac{W(x_{i})}{a}$ and $|z_{i}|\leq C_{2}\sqrt{\frac{W(x_{i})}{a}}$. \\
Moreover 
\begin{align*}
|y|_{1}=\sum_{i=1}^{n}|y_{i}|\leq\sum_{i=1}^{n} C_{1}\frac{W(x_{i})}{a}=C_{1}\frac{W_{n}(x)}{a}\leq C_{1}\frac{at}{a}=C_{1}t
\end{align*}
so $y\in C_{1}tB_{1}^{n}$ and
\begin{align*}
|z|_{2}=\sqrt{\sum_{i=1}^{n}z_{i}^2}\leq\sqrt{\sum_{i=1}^{n}C_{2}^{2}\frac{W(x_{i})}{a}}\leq\sqrt{C_{2}^{2}\frac{W_{n}(x)}{a}}\leq C_{2}\sqrt{\frac{at}{a}}=C_{2}\sqrt{t},
\end{align*}
so $z\in C_{2}\sqrt{t}B_{2}^{n}$, hence $x=y+z\in C_{1}tB_{1}^{n}+C_{2}\sqrt{t}B_{2}^{n}$.
\end{proof}

\begin{lem}
\label{lem:6.2}
The Cramer transform of a symmetric exponential distribution satisfies
\begin{align*}
\Lambda^{*}_{\nu}(x)\geq\Big(\sqrt{1+|x|}-1\Big)^{2}.
\end{align*}
\end{lem}

\begin{proof}
Denote $H(x)=\Lambda^{*}_{\nu}(x)-\Big(\sqrt{1+|x|}-1\Big)^{2}$. We need to show that $H\geq0$. Since $H$ is even we can assume that $x\geq0$. Standard computation gives
\begin{align*}
H'(x)=\frac{x}{1+\sqrt{x^{2}+1}}-1+\frac{1}{\sqrt{1+x}}.
\end{align*}
We will show that $H'\geq0$, from which using $H(0)=0$ the claim follows.\\
\ \\
\textbf{Case 1.} $(0<x\leq1)$\\
We compute
\begin{align*}
H'(x)=\frac{x}{1+\sqrt{1+x^2}}-1+\frac{1}{\sqrt{1+x}}\geq\frac{x}{1+\sqrt{1+x}}-1+\frac{1}{\sqrt{1+x}}=\frac{x(\sqrt{1+x}-1)}{(1+\sqrt{1+x})\sqrt{1+x}}\geq0.
\end{align*}
\textbf{Case 2.} $(x>1)$\\
Using $\sqrt{1+x^2}\leq\sqrt{2}-1+x$, we obtain
\begin{align*}
H'(x)&=\frac{x}{1+\sqrt{1+x^2}}-1+\frac{1}{\sqrt{1+x}}\geq\frac{x}{x+\sqrt{2}}-1+\frac{1}{\sqrt{1+x}}=\frac{x+\sqrt{2}-\sqrt{2}\sqrt{x+1}}{\sqrt{1+x}(x+\sqrt{2})}
\\
&=\frac{(\sqrt{x+1}-\frac{\sqrt{2}}{2})^{2}+\sqrt{2}-\frac{3}{2}}{\sqrt{1+x}(x+\sqrt{2})}\geq0,
\end{align*}
since 
\begin{align*}
\Big(\sqrt{x+1}-\frac{\sqrt{2}}{2}\Big)^{2}+\sqrt{2}-\frac{3}{2}\geq\Big(\sqrt{1+1}-\frac{\sqrt{2}}{2}\Big)^{2}+\sqrt{2}-\frac{3}{2}=\sqrt{2}-1>0.
\end{align*}
\end{proof}

\begin{proof}[Proof of Theorem \ref{tw:6}]
Let $n\geq1$ and $A\subset\mathbb{R}^n$ be such that $\nu^{n}(A)=\nu(-\infty;x]$. Using Proposition \ref{concentration} and Theorem \ref{tw:5} we obtain that 
\begin{align*}
\forall_{t>0}\quad\nu^{n}(A+B_{W_{n}}(2t))\geq\nu(-\infty;x+t], 
\end{align*}
where 
\begin{align*}
W_{n}(x)=\sum_{i=1}^{n}W(x_{i}),\quad W(x)=\Lambda_{\overline{\nu}}^{*}(x/2)=2\Lambda_{\nu}^{*}(x/4).
\end{align*}
To finish the proof it suffices to show that
\begin{align*}
\forall_{t>0}\quad B_{W_{n}}(2t)\subset 4tB_{1}^{n}+8\sqrt{t}B_{2}^{n}, 
\end{align*}
which by the Lemma \ref{lem:6.1} reduces to 
\begin{align*}
\forall_{t>0}\quad B_{W}(2t)\subset 4tB_{1}^{1}+8\sqrt{t}B_{2}^{1}.
\end{align*}
The last condition is equivalent to
\begin{align*}
\forall_{t>0}\forall_{x}\quad W(x)\leq2t\Longrightarrow\exists_{y}\quad|x-y|\leq 4t,y\leq8\sqrt{t}, 
\end{align*}
which follows from
\begin{align*}
\forall_{x}\exists_{y}\quad|x-y|\leq2W(x),y^{2}\leq32W(x).
\end{align*}
To finish the proof it suffices to show that
\begin{align*}
\forall_{x}\exists_{y}\quad \frac{1}{2}W(4x)=\Lambda^{*}_{\nu}(x)\geq\max\Big\{\Big|x-\frac{y}{4}\Big|,\Big(\frac{y}{8}\Big)^{2}\Big\}. 
\end{align*}
Since for $y(x)=8{\rm{sgn}}(x)(\sqrt{|x|+1}-1)$ we get
\begin{align*}
\Big|x-\frac{y(x)}{4}\Big|=\Big(\frac{y(x)}{8}\Big)^{2}=\Big(\sqrt{1+|x|}-1\Big)^{2},
\end{align*}
thus the last inequality follows from Lemma \ref{lem:6.2}.
\end{proof}

\section{Acknowledgement}
The results presented in the article were obtained during the author's MSc studies at the University of Warsaw under supervision of prof. Rafa{\l} Lata{\l}a


\begin{thebibliography}{9}

\bibitem{Bor} C.\ Borell, \emph{Convex set functions in d-space}, Period. Math. Hungar. 6 (1975), 111--136.

\bibitem{DeuStr} J.-D.\ Deuschel i D.W.\ Stroock, \emph{Large Deviations}, Pure Appl. Math. 137, Academic Press, Boston, MA (1989).

\bibitem{Goz} N.\ Gozlan, \emph{Characterization of Talagrand's like transportation-cost inequalities on the real line}, J.Funct.Anal. 250 (2007), 400--425.

\bibitem{Hen} D.\ Hensley, \emph{Slicing convex bodies-bounds for slice area in terms of the body's covariance}, Proc. Amer. Math. Soc. 79 (1980), 619--625.


\bibitem{LW} R.\ Lata{\l}a i J.O.\ Wojtaszczyk, \emph{On the infimum convolution inequality}, Studia Math. 189 (2008), 147--187.  

\bibitem{Led} M.\ Ledoux, \emph{The concentration of measure phenomenon}, Mathematical Surveys and Monographs 89, Amer. Math. Soc (2001).


\bibitem{MagTik} G.G.\ Magaril-llyaev i V.M.\ Tikhomirov, \emph{Convex Analysis: Theory and Applications}, Transl. Math. Monogr. 222, Amer. Math. Soc., Providence, RI (2003).

\bibitem{Mau} B.\ Maurey, \emph{Some deviation inequalities}, Geom. Funct. Anal. 1 (1991), 188--197.



\bibitem{Tal} M.\ Talagrand, \emph{A new isoperimetric inequality and the concentration of measure phenomenon}, Lecture Notes in Math. 1469, Springer, Berlin (1991), 94--124.

\end{thebibliography}
\end{document}